\documentclass[11pt,a4paper]{amsart}

\usepackage[T1]{fontenc}
\usepackage[utf8]{inputenc}

\usepackage{amsthm,amssymb,color,amscd}

\usepackage[english]{babel}

\usepackage[svgnames,hyperref]{xcolor}

\usepackage[pdfauthor={F. Mangolte \& C. Raffalli},  colorlinks, linktocpage, citecolor=red, linkcolor=red, urlcolor=red]{hyperref}

\newcommand{\RR}{\mathbb{R}}

\newcommand{\PP}{\mathbb{P}}
\newcommand{\pn}{\PP^n}
\newcommand{\ps}{\PP^*}

\newcommand{\xr}{X(\RR)}

\newtheorem{thm}{Theorem}
\newtheorem{lem}[thm]{Lemma}
\newtheorem{cor}[thm]{Corollary}
\newtheorem{prop}[thm]{Proposition}
\newtheorem{conj}[thm]{Conjecture}

\theoremstyle{definition}
\newtheorem{dfn}[thm]{Definition}

\begin{document}

\subjclass[2020]{14H50 14P05 14P25 52A10 52A15 52A20}
\keywords{real algebraic curve; bitangent; tritangent; support hyperplane; convex set}

\title[On a question of supports]{On a question of supports}

\author{Frédéric Mangolte and Christophe Raffalli}

\address{\emph{Fr\'ed\'eric Mangolte}
	\newline
	\textnormal{Aix Marseille Univ, CNRS, I2M, Marseille, France}
	\newline
	\textnormal{\texttt{frederic.mangolte@univ-amu.fr}}
	}
\address{\emph{Christophe Raffalli}
	\newline
	\textnormal{\texttt{christophe@raffalli.eu}}
	}

\thanks{Translated by Egor Yasinsky and Susanna Zimmermann from a preprint originally written in French.}

\begin{abstract}
We give a sufficient condition in order that $n$ closed connected subsets in the $n$-dimensional real projective space admit a common multitangent hyperplane.
\end{abstract}

\maketitle

\section{Introduction}

The motivation for the present note is a step in the proof of the following statements \cite[Corollary~5.5 and Theorem~6.1]{jm04} or \cite[§5.3]{ma-book,ma-book-en}:

\begin{thm}
\label{thm.dp2}
	Let $X$ be a real del Pezzo surface of degree 2 such that $X(\RR)$ is homeomorphic to the disjoint union of 4 spheres. Then a smooth map $f\colon \xr\to\mathbb{S}^2$ can be approximated by regular maps if and only if its topological degree is even.
\end{thm}

\begin{thm}
\label{thm.dp1}
	Let $X$ be a real del Pezzo surface of degree 1 such that $X(\RR)$ is homeomorphic to the disjoint union of 4 spheres and a projective plane. Then every smooth map $f\colon \xr\to\mathbb{S}^2$ can be approximated by regular maps.
\end{thm}
In the statements above $\mathbb{S}^2\subset \mathbb{R}^{3}$ is the real locus of the quadric $x_1^2+x_2^2+x_3^2=1$ and a \emph{regular map} is only regular on real algebraic loci, see \cite[Definitions~1.2.54 and 1.3.4]{ma-book,ma-book-en} for details.

One key point in the proof of the former statements was the existence of a bitangent line to any pair of connected components of a plane quartic and the existence of a tritangent conic to any triple of connected components of certain space sextic. To be precise we need the following:

\begin{prop}
\label{prop.multitangent}
Let $n=2,3$ and $X\subset \PP^n$ be a smooth real algebraic curve of degree $2n$ whose real locus $\xr$ has at least $n+1$ connected components. If $n=3$, assume furthermore that $X$ lies on a singular quadric.

Choose $n$ connected components $\Omega_1,\dots,\Omega_n$ of $\xr$.
Then there exists
a hyperplane of $\PP^n(\RR)$ which is tangent to $\Omega_i$ for all $1\leqslant i \leqslant n$.
\end{prop}

Given a pair of embedded circles in the plane, it seems rather clear that a line tangent to each of them exists provided that the circles are unnested.  Anyway, finding a rigorous proof of this is not straightforward and we did not find proper reference in the literature. It's less obvious to find a tritangent conic to three embedded circles in a cone. More generally, we can wonder how to generalize the obvious necessary condition to be unnested in a more general setting and, even better we can seek for a necessary and sufficient condition.
We find a sufficient (but still not necessary) condition in a rather general setting. This is the main result of this short note  (Theorem~\ref{thm.main}) from which we derive easily Proposition~\ref{prop.multitangent} as a particular case. Sections 2 and 3 are devoted to the proof of  this theorem. In Section 3, we prove Proposition~\ref{prop.multitangent} and propose a conjecture with a sufficient condition weaker than Theorem~\ref{thm.main}. We refer to the cited references for the proofs of Theorems~\ref{thm.dp2} and \ref{thm.dp1}.

\section{Some reminders}

We start with some well-known definitions from convex geometry.

\begin{dfn}[Convex hull]
	Let $E$ be an Euclidean space of dimension $n$. A subset $A\subset E$ is called {\it convex} in $E$ if and only if for all $x,y\in A$ and every $t\in [0,1]$ we have
	\[
	tx+(1-t)y\in A,
	\]
	i.e. the line segment joining $x$ and $y$ is contained in $A$. The {\it convex hull} of a subset $A\subset E$  is the smallest (in the inclusion sense) convex subset of $E$ containing $A$.
\end{dfn}

\begin{dfn}[Extremal point]
	Let $E$ be an Euclidean space of dimension $n$ and $A\subset E$ be a subset. We say that a point $x\in A$ is an {\it extremal point} of $A$ if the convex hull of $A\setminus \{x\}$ is still convex.
\end{dfn}

\begin{thm}[\bf Krein-Milman]
	Every non-empty compact convex subset of a Euclidean space admits an extremal point.
\end{thm}

\begin{proof}
See for instance \cite[Chap. II.4 Th. 1]{Bou53}.
\end{proof}

\begin{cor}
\label{cor}
Every non-empty compact subset of a euclidean space admits an extremal point.
\end{cor}

\begin{proof}
Let $A$ be a non-empty compact subset of a Euclidean space. Let $A_{c}$ be the convex hull of $A$. By Krein-Milman, there exists an extremal point $x\in A_{c}$.
If $x\notin A$, then the convex set $A_{c}\setminus \{x\}$ contains $A$ and it is a strict subset of $A_{c}$, which contradicts $A_c$ being the convex hull of $A$. Therefore, $x\in A$.
\end{proof}

\section{$n$-supporting hyperplanes}

\begin{dfn}[Supporting hyperplane]
Let $H$ be a hyperplane of a Euclidean space $E$ given by the equation $l(x)=a$, where $l$ is a linear form and $a\in \mathbb{R}$. We denote by $H^{+}$ and $H^-$ the half-spaces
\[
H^+:=\{x\in E\mid l(x)\geq a\} \qquad H^-:=\{x\in E\mid l(x)\leq a\}.
\]
Let $A\subset E$ be a subset of $E$ and $x\in A$. We say that $H$ is a {\em supporting hyperplane} of $A$ in $x$ (or that {\em $H$ leans on $A$ in $x$}) if and only if the following hold:
\begin{enumerate}
\item $x\in A\cap H$
\item $A\subset H^+$ or $A\subset H^-$.
\end{enumerate}

If $A$ is a subset of $\pn(\mathbb{R})$ and $x\in A$, we say that $H$ leans on $A$ in $x$ if and only if there exists an affine chart $E$ of $\pn(\mathbb{R})$ such that $x\in E$ and $H$ leans on $A$ in $x$ inside $E$.
\end{dfn}

\begin{dfn}[$r$-supporting hyperplane]
Let $A_1,\dots,A_r$ be subsets of $\pn(\mathbb{R})$. We say that $H$ is a {\em hyperplane of $r$-support} of $A_1,\dots,A_r$ if there exist points $x_1\in A_1,x_2\in A_2,\dots, x_r\in A_r$ such that $H$ is a supporting hyperplane of $A_i$ in $x_i$ for all $1\leq i\leq r$.
\end{dfn}

\begin{thm}
\label{thm.main}
Let $n\in \mathbb{N}$ and let $A_1,\dots,A_n\subset \pn(\mathbb{R})$ be closed connected subsets of $\pn(\mathbb{R})$. Suppose that there exists a point $p\in\pn(\mathbb{R})$ such that no hyperplane passing through $p$ meets all the $A_i$. Then there exists an $n$-supporting hyperplane of $A_1,\dots,A_n$.
\end{thm}

\begin{proof}
We write $\PP=\pn(\mathbb{R})$ and $\ps=(\pn(\mathbb{R}))^{*}$ for the dual projective space.
To each hyperplane $H\subset\PP$ given by an equation $\sum\lambda_kx_k=0$, we associate the point $H^{*}:=(\lambda_0:\lambda_1:\dots:\lambda_n)$ in $\ps$. To each point $q\in \PP$ we associate the dual hyperplane $q^{*}:=\{H^{*}\mid q\in H\}$ in $\ps$.

The hypothesis that there exists a point $p\in\PP$ such that no hyperplane passing through $p$ meets all the $A_i$ implies that the $A_i$ are pairwise disjoint.
Let $\mathcal{H}$ be the set of hyperplanes in $\PP$ that meet all the $A_i$. Since there is a hyperplane through $n$ points in $\PP$, we see that $\mathcal{H}$ is non-empty.
Let $\mathcal{H}^{*}$ be the image of $\mathcal{H}$ in the dual space $\ps$ via the above correspondance. Since $p^{*}$ corresponds to the set of hyperplanes in $\PP$ passing through $p$, the set $\mathcal{H}^{*}$ is contained in the complement of the hyperplane $p^{*}$ in $\ps$.
Let $U_p$ be the open affine complement of $p^{*}$ in $\ps$.

\begin{lem}
\label{lem.Ucompact}
The set $\mathcal{H}^{*}$ is compact in $U_p$.
\end{lem}
\begin{proof}
For each $1\leq i\leq n$, let $\mathcal{H}_i$ be the set of hyperplanes that meet $A_i$. We have $\mathcal{H}^{*}=\cap_{i=1}^n (\mathcal{H}_i)^{*}$.
The set $A_i$ being closed implies that $(\mathcal{H}_i)^{*}$ is closed. We start by showing that the complement of $\mathcal{H}^{*}$ in $U_p$ is open.

Indeed,
the natural map $\RR^{n+1}\setminus\{0\}\to\PP, (x_0,x_1,\dots,x_n)\mapsto [x_0:x_1:\dots:x_n]$ induces a continuous double cover $\mathbb{S}^n\rightarrow \PP$. The inverse image $B_i$  of $A_i$ through this map is a closed subset in the unit sphere of $\mathbb{R}^{n+1}$.
If $H$ is an hyperplane in $\PP$ that does not meet $A_i$, then its preimage $H'$ is an hyperplane in $\mathbb{R}^{n+1}$ which does not meet $B_i$. The intersection $H'\cap \mathbb{S}^n$ is the unit sphere of dimension $n-1$ in $H'$ and in particular is closed in $\mathbb{S}^n$.

If $d>0$ is the distance between the two compacts $B_i$ and $H'$, we can take $U_i$ the subset of $\ps$ formed by the duals of hyperplanes whose traces on $\mathbb{S}^n$ are at distance less than $\frac12$ of $B_i$. Then $U_i\setminus\{p\}$ is open in $U_p$.

 This shows that the complement of $(\mathcal{H}_i)^{*}$ in $\ps$ is open.
It follows that $\mathcal{H}^{*}$ is closed in $\ps$.
Moreover, the set $\mathcal{H}^{*}$ is bounded in $U_p$ because it is closed and $\mathcal{H}^{*}\cap p^*=\varnothing$.
Hence $\mathcal{H}^{*}$ is compact in $U_p$.
\end{proof}

By Corollary~\ref{cor} of Krein-Milman and Lemma~\ref{lem.Ucompact}, the set $\mathcal{H}^{*}$ admits an extremal point $H^*$. Let us show that $H$ is an $n$-supporting hyperplane of $A_1,\dots,A_n$.

We proceed by contradiction and without loss of generality, we can suppose that $H$ does not support $A_1$. Since $H\in\mathcal{H}$, there exists for each $i=2,\dots,n$ a point $y_i\in A_i\cap H$.
Let $P_1$ be a hyperplane passing through $p$ and $y_2,\dots,y_n$ and recall that $P_1$ does not meet $A_1$ by hypothesis.
Since $H$ does not lean on $A_1$, it does not lean on $A_1$ in the affine chart $E=\PP\setminus P_1$.
We place ourselves inside $E$.
The hyperplane $H\cap E$ defines two half-spaces $H^+$ and $H^-$ in $E$ and there exists $x_1\in A_1\cap H^+\setminus H$ and $x_2\in A_1\cap H^-\setminus H$.

Let $S$ be the closed segment $[x_1,x_2]$ in $E$. It intersects $H$. Let us show that
\begin{equation}\label{1}
\text{any hyperplane in $E$ that meets $S$ also meets $A_1$.}
\end{equation}
Let $P$ be a hyperplane of $E$ meeting $S$. If it meets $S$ in $x_1$ or $x_2$, we are finished. Suppose that $P\cap S\subset]x_1,x_2[$ and $A_1\cap P=\varnothing$. Let $O^+=P^+\setminus P$ and $O^-=P^-\setminus P$. The sets $O^+$ and $O^-$ are open subsets of $E$ and $A_1\subset O^+\cup O^-$. The subspace $A_1$ being connected in $E$, we have $A_1\subset O^+$ or $A_1\subset O^-$. This is impossible because $x_1\in O^+$ and $x_2\in O^-$ (or the other way around). this ends the proof of \eqref{1}.

Let $y\in S$. Since $y_2,\dots,y_n$ are pairwise distinct and are not contained in $E$ (remember that $y_i\in A_i\cap P_1$ for $i\in\{2,\dots,n\}$ by definition of $P_1$) and $S\subset E$, there exists a hyperplane $H_y\subset\PP$ through $y,y_2,\dots,y_n$. The hyperplane $H_y$ is contained in $\mathcal{H}$ because it meets $A_1$ by property \eqref{1}. 

The points $y_2,\dots,y_n$ define a line $D$ in $\ps$ and we have $(H_y)^{*}\in D$. Therefore, the set of $(H_y)^{*}$, $y\in S$, is a closed segment $S^{*}$. It is contained in $U_p$, because $p\notin H_y$, and $S^*$ is contained in $\mathcal{H}^{*}$ as a consequence of \eqref{1}.
Let $y_0=S\cap H$, where $H^*$ is the extremal point of $\mathcal{H}^*$ from above. Then $H^{*}=(H_{y_0})^{*}$ is a point in the interior of $S^{*}$. It  is therefore contained in the convex hull of $\mathcal{H}^{*}$ and cannot be an extremal point, because we lose convexity if we take it away. Hence the contradiction.
\end{proof}

\section{Conclusion}

\begin{proof}[Proof of Proposition~\ref{prop.multitangent}]
First recall that any hyperplane meets any connected component of $\xr$ in an even number of intersection points, counted with multiplicity, see e.g. \cite[Lemma~2.7.8]{ma-book,ma-book-en}. Let $p$ be a point of
$\xr\setminus\cup\Omega _i$. By definition of the degree, a hyperplane passing through $p$ cannot meet $n$ other components
of $\xr$ because $X$ has degree $2n$ in $\PP^n$.

The conclusion follows from Theorem~\ref{thm.main}.
\end{proof}

Theorem~\ref{thm.main} is enough to prove Proposition~\ref{prop.multitangent}, but it's easy to see that the existence of a point $p$ such that no hyperplane passing through $p$ meets all the $A_i$ is not necessary. Take for example two intersecting circles in the plane: as in Theorem~\ref{thm.main}, these are two subsets in the $2$-dimensional plane, but by any point $p$, there is a line meeting the two circles. Anyway, there is clearly a line tangent to them.

We propose the following conjecture using a weaker sufficient condition (which can be applied to the former example):

\begin{conj}
Let $\{A_i\}_{1\leq i\leq n}$ be closed connected subsets contained in an
affine subset of $\pn(\mathbb{R})$.  Let $C_i$ be the union of all
$(n-2)$-dimensional linear subspaces $P\subset \pn(\mathbb{R})$ such that for
all $j \neq i$, $1 \leq j \leq n$, $P$ meets the convex hull of $A_j$. Assume
that for all $1 \leq i \leq n$, $A_i$ is not included in interior of $C_i$,
then there exists an $n$-supporting hyperplane of $A_1,\dots,A_n$.
\end{conj}

Remark that this new sufficient condition is still unnecessary: consider three
disjoint spheres $A_1$, $A_2$ and $A_3$ with the same radius and whose center
are on the same line.  If $A_1$ is not the sphere in the middle it is in the
interior of the union of all lines meeting $A_2$ and $A_3$.

We can see that the sufficient condition of the conjecture is weaker than the
one of Theorem~\ref{thm.main}, by contraposition. If the condition of the
conjecture is not satisfied, then there exists $i$ such that $A_i$ is included
in the interior of the union of the $(n-2)$-dimensional linear subspaces
meeting each convex hull of $A_j$, $j\ne i$.
Then there exists a $(n-2)$-dimensional linear
subspace $P$ meeting each convex hull of $A_i$. Let $p\in \PP$, then the hyperplane generated
by $p$ and $P$ meet each  convex hull of $A_i$, hence each $A_i$ as they are connected,
which contradicts the condition of the theorem.

\medskip
We could also ask about the number of multi-tangent planes.

\begin{prop}
Under the conditions of Theorem~\ref{thm.main}, if each $A_i$ contains a non empty open subset, then there is at least $n+1$ distinct a $n$-supporting hyperplanes of $A_1,\dots,A_n$.
\end{prop}

\begin{proof}
If each $A_i$ contains a non empty open subset, so does
$\mathcal{H}^*$. This implies that there is at least $n+1$ distinct extremal
points for $\mathcal{H}^*$. Indeed, if $\mathcal{H}^*$ as less than $n+1$ extremal points, it is
the convex-hull of its extremal points and therefore it is an hyperplane of
dimension at most $n-1$ hence does not contain any open set.
Then, the proof of theorem~\ref{thm.main} establishes that each extremal
points for $\mathcal{H}^*$ corresponds to a distinct $n$-supporting
hyperplanes.
\end{proof}

However, it seems that the conditions of this theorem implies that we have $2^n$ extremal points
(in dimension $2$: $4$ bitangent lines, $8$ in dimension~$3$, etc.) By going
either below or above each $A_i$. This suggest that $\mathcal{H}^*$ ressemble to a cube.
Moreover, all the examples we studied lead us to propose the following conjecture.
\begin{conj}
The main condition of Theorem~\ref{thm.main} is sufficient and necessary to have
$2^n$ multi-tangent planes when the $A_i$ are not thin (i.e. contain an open subset).
\end{conj}

\bibliographystyle{amsalpha}
\bibliography{biblio-tvar,biblio-perso}

\end{document}